\documentclass[12pt]{amsart}
\usepackage{amsmath}
\usepackage{amsthm}
\usepackage{amsfonts}
\usepackage{amssymb}
\newtheorem{Theorem}{Theorem}[section]
\newtheorem{Proposition}[Theorem]{Proposition}
\newtheorem{Lemma}[Theorem]{Lemma}

\theoremstyle{definition}
\newtheorem{Definition}{Definition}[section]
\theoremstyle{remark}

\numberwithin{equation}{section}

\newcommand{\R}{{\mathbb R}}

\newcommand{\N}{{\mathbb N}}

\begin{document}

\title[Absolutely continuous spectrum]{The absolutely continuous spectrum
of one-dimensional Schr\"odinger operators}

\author{Christian Remling}

\address{Mathematics Department\\
University of Oklahoma\\
Norman, OK 73019}

\email{cremling@math.ou.edu}

\urladdr{www.math.ou.edu/$\sim$cremling}

\date{October 29, 2007; revised: January 31, 2008}



\thanks{to appear in \textit{Math.\ Phys.\ Anal.\ Geom.}}
\begin{abstract}
This paper deals with general structural properties of one-dimensional
Schr\"odinger operators with some absolutely continuous spectrum.
The basic result says that the $\omega$ limit points of the potential
under the shift map are reflectionless on the support of the absolutely
continuous part of the spectral measure.
This implies an Oracle Theorem for such potentials and Denisov-Rakhmanov
type theorems.

In the discrete case, for Jacobi operators,
these issues were discussed in my recent paper \cite{Remac}.
The treatment of the continuous case in the present paper depends
on the same basic ideas.
\end{abstract}
\maketitle
\section{Introduction}
This note discusses basic properties of one-dimensional Schr\"odinger operators
\[
H= -\frac{d^2}{dx^2} + V(x)
\]
with some absolutely continuous spectrum.
It is a supplement to my recent paper \cite{Remac}. In \cite{Remac}, I
dealt with the discrete case exclusively. As one would expect, the basic
ideas that were presented in \cite{Remac} can also be used to analyze
the continuous case. It is
the purpose of this paper to give such a treatment; basically, this will be
a matter of making the appropriate definitions.

Therefore, my general philosophy will be to keep this note brief.
I will assume that
the reader is familiar with at least the general outline of the discussion
of \cite{Remac} and only focus on those aspects where the extension to
the continuous case is perhaps not entirely obvious. By the same token,
I will not say much about related work here; please see again \cite{Remac}
for a fuller discussion.

Given a potential $V$,
we will consider limit points $W$ under the shift $(S_xV)(t)=V(x+t)$,
as $x\to\infty$. We will thus need a suitable topology on a suitable space of potentials.
This will naturally lead us to consider generalized Schr\"odinger operators, with
measures as potentials.

The basic result, from which everything else will follow,
is Theorem \ref{T4.1} below.
It says that the limits $W$ are necessarily reflectionless
(this notion will be defined later) on the support of the absolutely continuous part
of the spectral measure. This is a very strong condition; it severely restricts the
structure of potentials with some absolutely continuous spectrum.
As in \cite{Remac}, this result crucially depends on earlier work
of Breimesser and Pearson \cite{BP1,BP2}.

We will present two applications of Theorem \ref{T4.1} here; both
are analogs of results from \cite{Remac}. The first application
gives an easy and transparent proof of a continuous Denisov-Rakhmanov
\cite{Dencont,Den,Rakh} type theorem.

We denote by $\Sigma_{ac}$
the essential support of the absolutely continuous part of the spectral measure;
this is determined up to sets of (Lebesgue) measure zero. If
we write $\rho$ for the
spectral measure, we can define (a representative of)
$\Sigma_{ac}$ as the set where
$d\rho/dt > 0$. The absolutely continuous spectrum, $\sigma_{ac}$, may be
obtained from $\Sigma_{ac}$ by taking the essential closure. The essential
spectrum, $\sigma_{ess}$, can be defined as the set of accumulation points of
the spectrum.
\begin{Theorem}
\label{TDR}
Let $V$ be a uniformly locally integrable (half line) potential (that is, we assume
that $\sup_n \int_n^{n+1} |V(x)|\, dx < \infty$). Suppose that
\[
\sigma_{ess}=\Sigma_{ac} = [0, \infty ) .
\]
Then
\[
\lim_{x\to\infty} \int V(x+t)\varphi(t)\, dt = 0
\]
for every continuous $\varphi$ of compact support.
\end{Theorem}
Denisov proved this earlier \cite[Theorem 2]{Dencont}, under the somewhat stronger
assumption that $V$ is bounded.

The conclusion of Theorem \ref{TDR}
says that $V(x)$ tends to zero as $x\to\infty$ in weak $*$ sense
(more precisely, it is the sequence of \textit{measures} $V(x+t)\, dt$ that converges).
It will become clear later that this mode of convergence is natural here. Also,
examples of the type $V=U^2+U'$ with a rapidly decaying, but oscillating $U$
show that stronger modes of convergence of $V$ can not be expected.

Theorem \ref{TDR} will be proved in Section 4.
As in \cite[Theorem 1.8]{Remac}, it should
be possible to use the same technique to establish an analogous result for
finite gap potentials and spectra (and beyond), but we will not pursue this theme here.

Let us now discuss a second structural consequence of Theorem \ref{T4.1};
in \cite{Remac}, I introduced the designation \textit{Oracle Theorem}
for statements of this type. The Oracle Theorem
says that for operators with absolutely continuous spectrum, it is possible to
approximately predict future values of the potential, with arbitrarily high accuracy,
based on information about past values.

The precise formulation will involve measures $\mu$ as potentials and some
additional technical devices; these
will of course be explained in more detail later. To get a preliminary impression
of what the Oracle Theorem is saying, it is possible to replace $\mu$ by
a (uniformly locally integrable) potential $V$ in Theorem \ref{OT} below.

We will work with spaces $\mathcal V^C_J$ of
signed Borel measures $\mu$ on intervals $J$. For now, we can pretend that
a measure $\mu$ is in $\mathcal V^C_J$ if $|\mu|(J)\le C|J|$, but, for inessential
technical reasons, the actual definition will be slightly different.
If endowed with the weak $*$ topology,
these spaces $\mathcal V^C_J$ are compact and in fact metrizable.
The metric $d$ that is used below arises in this way. We will also use a
similarly defined space $\mathcal V^C$ of measures on $\R$. See Section 2
for the precise definitions.

Finally, $S_x\mu$ will denote the shift by $x$
of the measure $\mu$, that is,
\begin{equation}
\label{Sxmu}
\int f(t)\, d(S_x\mu)(t) = \int f(t-x)\, d\mu(t) .
\end{equation}
If $d\mu = V\, dt$ is a locally integrable potential $V$, then this reduces to
the shift map $(S_xV)(t)=V(x+t)$ that was introduced above.
\begin{Theorem}[The Oracle Theorem]
\label{OT}
Let $A\subset\R$ be a Borel set of positive (Lebesgue) measure, and let
$\epsilon>0$, $a, b\in\R$ ($a<b$), $C>0$. Then there exist $L>0$ and a continuous function
(the {\em oracle})
\[
\Delta : \mathcal V^C_{(-L,0)} \to \mathcal V^C_{(a,b)}
\]
so that the following holds. If $\mu\in\mathcal V^C$ and the half line operator
associated with $\mu$ satisfies $\Sigma_{ac}\supset A$, then there exists an
$x_0>0$ so that for all $x\ge x_0$, we have that
\[
d\left( \Delta\left( \chi_{(-L,0)}S_x\mu \right) , \chi_{(a,b)} S_x\mu \right)
< \epsilon .
\]
\end{Theorem}
In other words, for large enough $x$, we can approximately determine
the potential on $(x+a,x+b)$ from its values on $(x-L,x)$, and the function (oracle)
that does this prediction is in fact independent of the potential.
Moreover, by adjusting $a,b$, we can
also specify in advance how far the oracle should look into the future.
\section{Topologies on spaces of potentials}
We need a topology on a suitable set of potentials that makes this space compact
and also interacts well with other basic objects such as $m$ functions. This is
easy to do if we are satisfied with working with potentials that obey a local
$L_p$ condition with $p>1$. Indeed, for every $p>1$ (and $C>0$), we can define
\[
\mathcal V_p^C = \left\{ V:\mathbb R \to\mathbb R : \int_n^{n+1} |V(x)|^p\, dx \le C^p
\textrm{ for all }n\in\mathbb Z \right\} .
\]
Closed balls in $L_p$ are compact in the weak $*$ topology if $p>1$; in fact, these
compact topological spaces are metrizable. Pick such metrics $d_n$; in other words,
if $W_j, W\in L_p(n,n+1)$, $\|W_j\|_p, \|W\|_p \le C$,
then $d_n(W_j,W)\to 0$ precisely if $W_j\to W$ in the
weak $*$ topology, that is, precisely if
\[
\int_n^{n+1} W_j(x)g(x)\, dx \to \int_n^{n+1} W(x)g(x)\, dx
\quad\quad (j\to\infty)
\]
for all $g\in L_q(n,n+1)$, where $1/p+1/q = 1$. Then, using these metrics, define,
for $V,W\in\mathcal V_p^C$
\[
d(V,W) = \sum_{n=-\infty}^{\infty} 2^{-|n|} \frac{d_n(V_n,W_n)}{1+d_n(V_n,W_n)} ;
\]
here $V_n, W_n$ denote the restrictions of $V,W$ to $(n,n+1)$.

This metric generates the product topology on $\mathcal V_p^C$, where this space
is now viewed as the product of the closed balls of radius $C$ in $L_p(n,n+1)$.
In particular, $(\mathcal V_p^C , d)$ is a compact metric space.

This simple device allows us to establish continuous analogs of the results of
\cite{Remac} without much difficulty at all, but it is unsatisfactory
because the most natural and general local condition on the potentials
is an $L_1$ condition. Since $L_1$ is not a dual space, we will then
need to consider measures to make an analogous approach work. Thus we define
\[
\mathcal V^C = \left\{ \mu\in\mathcal M (\mathbb R) : |\mu|(I) \le C \max \{ |I|, 1 \}
\textrm{ for all intervals } I\subset\mathbb R \right\} .
\]
Here, $\mathcal M (\mathbb R)$ denotes the set of (signed) Borel measures on $\mathbb R$.
We can now proceed as above to define a metric on $\mathcal V^C$: Pick a countable
dense (with respect to $\|\cdot \|_{\infty}$) subset $\{ f_n: n\in\N\}\subset
C_c(\mathbb R)$, the continuous functions of compact support, and put
\[
\rho_n(\mu,\nu) = \left| \int f_n(x)\, d(\mu-\nu)(x) \right| .
\]
Then define the metric $d$ as
\begin{equation}
\label{defd}
d(\mu,\nu) = \sum_{n=1}^{\infty} 2^{-n} \frac{\rho_n(\mu,\nu)}{1+\rho_n(\mu,\nu)} .
\end{equation}
Clearly, $d(\mu_j,\mu)\to 0$ if and only if
\[
\int f(x)\, d\mu_j(x) \to \int f(x)\, d\mu(x)\quad\quad (j\to\infty)
\]
for all $f\in C_c(\mathbb R)$. Moreover, $(\mathcal V^C, d)$ is a compact space.
To prove this, let $\mu_n\in\mathcal V^C$. By the Banach-Alaoglu Theorem, closed
balls in $\mathcal M([-R,R])$ are compact. Use this and a diagonal process to find
a subsequence $\mu_{n_j}$ with the property that
$\int f\, d\mu_{n_j}\to\int f\,d\mu$ for all $f\in C_c(\mathbb R)$, for some
$\mu\in\mathcal M(\R)$. The proof can now be completed by noting that a
measure $\nu\in\mathcal M(\mathbb R)$ is in $\mathcal V^C$ if and only if
\[
\left| \int f(x)\, d\nu(x) \right| \le C\max\{ \textrm{diam}(\textrm{supp }f), 1 \} \|f\|_{\infty}
\]
for all $f\in C_c(\R)$.

The same construction can be run if $\mathbb R$ is replaced by an interval $J$,
and these spaces, which we will denote by $\mathcal V^C_J$, will also play an
important role later on.
\section{Schr\"odinger operators with measures}
We are thus led to consider Schr\"odinger operators with measures as
potentials; therefore, we must now clarify what the precise meaning of this
object is. There is, of course, a considerable amount of previous work
on these issues; see, for example, \cite{AGHKH,BARem,BEKS,BFT} and the references
cited therein. Here, we will follow the approach
of \cite{BARem}. Actually, Schr\"odinger \textit{operators} will not play a
central role in this paper, at least not explicitly. Therefore, we will only indicate how
to make sense out of the Schr\"odinger \textit{equations}
$-f''+\mu f=zf$. We can then use these to define Titchmarsh-Weyl $m$ functions,
spectral measures etc., and we refer the reader to \cite{BARem} for the
(straightforward) definition of domains that yield self-adjoint operators.

There are two obvious attempts, and these conveniently lead to the same result: If
$I\subset\mathbb R$ is an open interval and $f\in C(I)$,
we can call $f$ a solution to the Schr\"odinger equation
\begin{equation}
\label{se}
-f'' + f\mu = zf
\end{equation}
if \eqref{se} holds in the sense of distributions on $I$. Alternatively,
one can work with the \textit{quasi-derivative}
\[
(Af)(x) = f'(x) - \int_{[0,x]} f(t)\, d\mu(t) ;
\]
if $x<0$, then $\int_{[0,x]}$ needs to be replaced with $-\int_{(x,0)}$ here. We now say that
$f$ solves \eqref{se} on $I$ if both $f$ and $Af$ are (locally)
absolutely continuous and $-(Af)'=zf$ on $I$. This
new definition is motivated
by the observation that, at least formally, $(Af)'=f''-f\mu$.

A slight modification of the argument from the proof of \cite[Theorem 2.4]{BARem} then
shows that this latter interpretation of \eqref{se} is
equivalent to the equation holding in $\mathcal D'(I)$. The basic observation here is that
if $f$ is continuous, then
$(Af)' = f'' - f\mu$ in $\mathcal D'$, not only formally.

Note that if $f$ solves \eqref{se}, then $f'$ is of bounded variation and the
jumps can only occur at the atoms of $\mu$.

If $\mu\in\mathcal V^C$, we have limit point case at both endpoints. This means
that for $z\in\mathbb C^+$ (the upper half plane in $\mathbb C$), there exist unique
(up to a factor) solutions $f_{\pm}(x,z)$ of \eqref{se} on $\mathbb R$
satisfying $f_-\in L_2(-\infty,0)$,
$f_+\in L_2(0,\infty)$. The Titchmarsh-Weyl $m$ functions of the problems on
$(-\infty,x)$ and $(x,\infty)$, with Dirichlet boundary conditions at
$x$ ($u(x)=0$), are now defined as follows:
\begin{equation}
\label{defM}
m_{\pm}(x,z) = \pm \frac{f_{\pm}'(x,z)}{f_{\pm}(x,z)}
\end{equation}
We will use this formula only for points $x$ with $\mu(\{x\})=0$ so that the possible
discontinuities of $f'$ cannot cause any problems here.

\begin{Definition}
\label{D3.1}
Let $A\subset\mathbb R$ be a Borel set.
We call a potential $\mu\in\mathcal V^C$ \textit{reflectionless} on $A$ if
\begin{equation}
\label{3.2}
m_+(x,t)= -\overline{m_-(x,t)} \quad\quad \textrm{for almost every }t\in A
\end{equation}
for some $x\in\mathbb R$ with $\mu(\{ x\} )=0$.

The set of reflectionless potentials $\mu\in\bigcup_{C>0}\mathcal V^C$ on $A$ is
denoted by $\mathcal R(A)$.
\end{Definition}
This is a key notion for everything that follows. If we have \eqref{3.2} for
some $x$, then we automatically get this equation at \textit{all} points of
continuity of $\mu$. Moreover, the exceptional set implicit in \eqref{3.2} can
be taken to be independent of $x$. To prove these remarks, observe that if
$m_{\pm}(x,t)\equiv\lim_{y\to 0+} m_{\pm}(x,t+iy)$ exists for some $x,t\in\mathbb R$,
then this limit exists for all $x$ (and the same $t$). Moreover, as a function
of $x$, the $m$ functions are of bounded variation and (using distributional derivatives)
\[
\pm \frac{d}{dx} m_{\pm} = \mu - z - m^2_{\pm} .
\]
The claim now follows by considering $(d/dx)(m_++\overline{m_-})(x,t)$.

The other key notion is that of the $\omega$ limit set of a potential
$\mu\in\mathcal V^C$ under the shift map. This was already mentioned in the
introduction, and we can now give the more precise definition
\[
\omega(\mu) = \bigl\{ \nu\in\mathcal V^C:
\textrm{ There exist $x_n\to\infty$ so that }d(S_{x_n}\mu, \nu ) \to 0 \bigr\} .
\]
For the definition of the shifted measures $S_x\mu$, see \eqref{Sxmu}.
Typically, $\mu$ will be given as a half line potential $V$, but it is of course easy
to interpret $V$ as an element $d\mu = V\, dx$ of $\mathcal V^C$ (such a $\mu$ automatically
gives zero weight to $(-\infty,0]$).

The compactness of $\mathcal V^C$ ensures that $\omega(\mu)$ is non-empty, compact,
and invariant under $\{ S_x: x\in\mathbb R\}$. Moreover, and in contrast to the
discrete case, $\omega(\mu)$ is also connected because we now have a \textit{flow} $S_x$.
\section{Main results and their proofs}
It will be useful to introduce $\mathcal V^C_+$
as the set of all $\mu\in\mathcal V^C$
with $|\mu|((-\infty,0])=0$. These measures will serve as the potentials
of half line problems on $(0,\infty)$. We can think of such a $\mu$ as a measure on $(0,\infty)$
or on $\R$.
\begin{Theorem}
\label{T4.1}
Let $\mu\in\mathcal V^C_+$. Then $\omega(\mu) \subset \mathcal R(\Sigma_{ac})$.
\end{Theorem}
Here, $\Sigma_{ac}$ denotes an essential support of the absolutely continuous
part of the spectral measure of the \textit{half line} problem on $(0,\infty)$ (say).

Theorem \ref{T4.1} is proved in the same way as the analogous result (Theorem 1.4) from \cite{Remac}.
Therefore, we will only make a few quick remarks and then leave the matter at that.

First of all, note that although the original result of Breimesser and Pearson
\cite[Theorem 1]{BP1} is formulated for Schr\"odinger operators with locally integrable
potentials, the same proof also establishes the result for operators with measures as
potentials. Indeed, one never works with the potential itself but only with solutions
to the Schr\"odinger equation \eqref{se} or with transfer matrices. See also \cite[Appendix A]{Remac}.

As a second ingredient, we need continuous dependence of the
(half line) $m$ functions $m_{\pm}$ on the potential.
\begin{Lemma}
\label{L4.1}
Let $\mu_n,\mu\in\mathcal V^C$ and suppose that $d(\mu_n,\mu)\to 0$. Fix
$x\in\mathbb R$ with $\mu_n(\{x\})=\mu(\{ x\})=0$. Then
\[
m_{\pm}(x,z;\mu_n) \to m_{\pm}(x,z;\mu) ,
\]
uniformly on compact subsets of $\mathbb C^+$.
\end{Lemma}
This follows because convergence in $\mathcal V^C$ implies weak $*$ convergence
of the restrictions of the measures to compact intervals, at least if the endpoints
of these intervals are not atoms of $\mu$. It then follows that
the solutions to the Schr\"odinger
equation converge, locally uniformly in $z$. This is most conveniently established by rewriting
the Schr\"odinger equation as an integral equation. See, for example, \cite[Lemma 6.3]{BARem}
for more details. One can now use \eqref{defM} to obtain the Lemma. In fact, it is also
helpful to approximate $m_{\pm}$ by $m$ functions of problems on \textit{bounded} intervals.
This allows us to work with solutions that satisfy a fixed initial condition.

To prove Theorem \ref{T4.1}, fix $\nu\in\omega(\mu)$. By definition of the
$\omega$ limit set, there exists a sequence $x_j\to\infty$ so that $S_{x_j}\mu\to\nu$
in $(\mathcal V^C , d)$. Fix $x\in\mathbb R$ with $\mu(\{x+x_j\})=\nu(\{ x\})=0$. By Lemma \ref{L4.1},
\[
m_{\pm}(x+x_j,z;\mu) \to m_{\pm}(x,z;\nu) \quad\quad (j\to\infty) ,
\]
locally uniformly in $z$. The Breimesser-Pearson Theorem \cite[Theorem 1]{BP1} (see also
\cite[Theorem 3.1]{Remac}) together with \cite[Theorem 2.1]{Remac}
then yield a relation between $m_+(x,z;\nu)$ and
$m_-(x,z;\nu)$ which turns out to be equivalent to the condition from Definition \ref{D3.1},
with $A=\Sigma_{ac}$. This last part of the argument is identical with the corresponding
treatment of \cite{Remac}.

Honesty demands that I briefly comment on a technical (and relatively insignificant point)
here: To run the argument in precisely this form, one needs a
slight modification of either Lemma \ref{L4.1} or the original Breimesser-Pearson Theorem.
The easiest solution would be to prove the Breimesser-Pearson Theorem for two
half line $m$ functions $m_{\pm}$ (in the original version
from \cite{BP1,BP2}, $m_-$ refers to a bounded interval).
Alternatively, one can use a variant of Lemma \ref{L4.1} where the approximating $m$
functions may be associated with bounded (but growing) intervals.

Let us now show how Theorem \ref{T4.1} can be used to produce Denisov-Rakhmanov type theorems.
We will automatically obtain the following slightly more general version of Theorem \ref{TDR}.
\begin{Theorem}
\label{T4.2}
Let $\mu\in\mathcal V^C_+$, and
suppose that the half line operator generated by $\mu$ satisfies
\[
\sigma_{ess} = \Sigma_{ac} = [0,\infty) .
\]
Then $d(S_x\mu,0)\to 0$ as $x\to\infty$.
\end{Theorem}
The proof will also depend on the following observation (whose discrete analog was pointed
out in \cite{LS}; see also \cite{LS2}).
\begin{Proposition}
\label{P4.1}
Let $\mu\in\mathcal V^C_+$ and assume that $\nu\in\omega(\mu)$. Then
\[
\sigma(\nu) \subset \sigma^+_{ess}(\mu) .
\]
\end{Proposition}
Here, $\sigma(\nu)$ is the spectrum of $-d^2/dx^2+\nu$ on $L_2(\R)$, while
$\sigma^+_{ess}(\mu)$ denotes the essential spectrum of the \textit{half line} operator
$-d^2/dx^2+\mu$ on $L_2(0,\infty)$ (say).
\begin{proof}
In the discrete case, this followed from a quick argument using
Weyl sequences. In the continuous case, this device is not as easily implemented
because of domain questions. The following alternative argument avoids these issues and
thus seems simpler: Suppose that
$d(S_{x_j}\mu, \nu)\to 0$. Then the
whole line (!) operators associated with $S_{x_j}\mu$ converge in strong resolvent
sense to the (whole line) operator generated by $\nu$. To prove this fact, one can argue as
in Lemma \ref{L4.1} above.

Since the operators with shifted potentials
are unitarily equivalent to the operator generated by $\mu$ itself, it follows from
\cite[Theorem VIII.24(a)]{RS1} that $\sigma(\nu)\subset\sigma(\mu)$. The $\omega$ limit
set does not change if $\mu$ is modified on a left half line; any discrete eigenvalue, however,
can be moved (or removed) by such a modification. Similarly,
$\sigma_{ess}=\sigma^+_{ess}\cup\sigma^-_{ess}$, and $\sigma^-_{ess}$ is completely at our
disposal, so we actually obtain the stronger claim
of the Proposition.
\end{proof}
\begin{proof}[Proof of Theorem \ref{T4.2}]
We will show that $\omega(\mu)$ consists of the zero potential only. This will imply
the claim because the distance between $S_x\mu$ and $\omega(\mu)$ must go to zero as $x\to\infty$.

By Theorem \ref{T4.1} and Proposition \ref{P4.1}, any $\nu\in\omega(\mu)$ must satisfy
\begin{equation}
\label{4.3}
\sigma(\nu)=[0,\infty), \quad \nu\in\mathcal R((0,\infty)) .
\end{equation}
Here, we use the (well known) fact that $\textrm{Im }m_{\pm}>0$ almost everywhere on $A$
if the corresponding potential is reflectionless on this set. Indeed, \eqref{3.2} shows
that otherwise we would have $m_++m_-=0$ on a set of positive measure, hence everywhere,
but this is clearly impossible.

What we will actually prove now is that only the zero potential, $\nu=0$,
satisfies \eqref{4.3}. This argument follows a familiar pattern; see, for example,
\cite{CGHL} or \cite{SiSt} (especially Lemma 4.6 and the discussion that follows)
for similar arguments in somewhat different situations.

Suppose that $\nu\in\mathcal V^C$ obeys \eqref{4.3}, and fix $x\in\R$ with
$\nu(\{x\})=0$. Let $m_{\pm}$ be the $m$ functions of $H_{\nu}=-d^2/dt^2+\nu$ on
$L_2(x,\infty)$ and $L_2(-\infty,x)$, respectively,
and consider the function
\[
H(z) = m_+(x,z)+m_-(x,z) = -\frac{W(f_+,f_-)}{f_+(x,z)f_-(x,z)} .
\]
Here, $W(u,v)=uv'-u'v$ denotes the Wronskian. This last expression identifies $H$ as
\[
H(z) = - \frac{1}{G(x,x;z)} ,
\]
the negative reciprocal of the (diagonal of the) Green function of $H_{\nu}$. Compare, for example,
\cite[Section 9.5]{CodLev}. (Of course, this reference does not discuss Schr\"odinger
operators with measures, but the rather elementary argument based on the variation of
constants formula generalizes without any difficulty.)

The defining property of $G$ is given by
\[
\left( (H_{\nu}-z)^{-1}\varphi \right)(x) = \int_{-\infty}^{\infty} G(x,y;z)\varphi(y)\, dy .
\]
This holds for $z\notin\sigma(\nu)=[0,\infty)$, $\varphi\in L_2(\R)$. The spectral theorem
shows that if $z=-t<0$, then
\[
\langle \varphi, (H_{\nu}+t)^{-1} \varphi \rangle = \int_{[0,\infty)}
\frac{d\|E(s)\varphi\|^2}{s+t} > 0 .
\]
Since $G$ is continuous in $x,y$, this implies that
$G(x,x,t)\ge 0$, thus $H(t)<0$ for $t<0$. Furthermore,
the fact that $\nu\in\mathcal R((0,\infty))$ implies that $\textrm{Re }H(t)=0$ for
almost every $t>0$.

So we know the phases of the boundary
values of the Herglotz function $H$ almost everywhere.
By the exponential Herglotz representation
(or, synonymously, the Herglotz representation of $\ln H(z)$), this determines $H$
up to a (positive) multiplicative constant. Since $H_0(z)=(-z)^{1/2}$ has the properties
described above, this says that for suitable $c>0$,
\begin{equation}
\label{Hz}
H(z) = c\sqrt{-z} = -\frac{c}{\sqrt{2}}+
\frac{c}{\pi} \int_0^{\infty}\left( \frac{1}{t-z} - \frac{t}{t^2+1}\right)
t^{1/2}\, dt .
\end{equation}
We will now need some information on the large $z$ asymptotics of $m$ functions.
This subject has been analyzed in considerable depth; see, for example, \cite{Atk,GS,HKS,Ryb}.
Of course, the treatment of these references needs to be adjusted here to cover the
case of Schr\"odinger operators with measures, but this is easy to do, especially since
we will only need the rather unsophisticated estimate
\[
m_{\pm}(x,-\kappa^2) = -\kappa + o(1) \quad\quad (\kappa\to\infty) .
\]
Here, it is important that we assumed that $\nu(\{ x\} )=0$.

Since $H=m_++m_-$, it now follows that $c=2$ in \eqref{Hz}.
Furthermore, the measures from the Herglotz representations
of $m_{\pm}$ are absolutely continuous with respect to the measure
\[
d\rho(t)= \frac{2}{\pi} \chi_{(0,\infty)}(t) t^{1/2}\, dt
\]
from the Herglotz representation \eqref{Hz} of $H$. In fact, we can write
\[
m_{\pm}(x,z) = A_{\pm} + \int \left( \frac{1}{t-z} - \frac{t}{t^2+1}\right)
g_{\pm}(t)\, d\rho(t) ,
\]
with $A_{\pm}\in\R$, $A_++A_-=-\sqrt{2}$,
and, more importantly, $0\le g_{\pm} \le 1$ and $g_+ + g_- =1$. More can be said here:
Since $\nu$ is reflectionless on $(0,\infty)$, we can use \eqref{3.2} to deduce
that $\textrm{Im }m_+(x,t)=
\textrm{Im }m_-(x,t)$ for almost every $t>0$. But for
almost every $t>0$, we have that $\textrm{Im }m_{\pm}(x,t)= g_{\pm}(t) (2/\pi) t^{1/2}$,
thus $g_+=g_-=1/2$ almost everywhere.
It now follows that
\[
m_{\pm}(x,z) = \sqrt{-z} .
\]
But $m_0=\sqrt{-z}$ is the $m$ function for
zero potential, thus $\nu=0$, as desired.
This last step is a basic result in inverse spectral theory for potentials
($m$ determines $V$); here, we
of course need a version for measures, but this extension poses no difficulties. See,
for instance, \cite[Theorem 6.2(b)]{BARem} (this needs to be combined with the fact
that $m$ determines $\phi$, but this is also discussed in \cite{BARem}).
\end{proof}
It remains to prove the Oracle Theorem. We prepare for this by making a couple
of new definitions. First of all, put
\[
\mathcal R^C(A) = \mathcal R(A) \cap \mathcal V^C .
\]
Next, we consider again spaces of half line potentials, and we now think
of these as restrictions of measures
$\mu\in\mathcal V^C$:
\begin{align*}
\mathcal V^C_+ & = \{ \chi_{(0,\infty)}\mu : \mu\in\mathcal V^C \}, \\
\mathcal V^C_- & = \{ \chi_{(-\infty,0)}\mu: \mu\in\mathcal V^C \}
\end{align*}
I emphasize that on $\mathcal V^C_{\pm}$, we do \textit{not} use the topology that is
induced by $\mathcal V^C\supset\mathcal V^C_{\pm}$. That would quite obviously be a bad idea
because it would make the restriction map $\mu\mapsto \chi_{(0,\infty)}\mu$ discontinuous;
consider, for example, the sequence $\mu_n = \delta_{1/n}$. Instead, we just observe
that in the notation from Section 2, we can identify
$\mathcal V^C_+=\mathcal V^C_J$, where $J=(0,\infty)$,
and we use the topology and metric described in Section 2. In other words, if
we denote this metric by $d_+$, then $d_+(\mu_n,\mu)\to 0$ if and only if
\[
\int f(x)\, d\mu_n(x) \to \int f(x) \, d\mu(x) \quad\quad (n\to\infty)
\]
for all continuous $f$ whose support is a compact subset of $(0,\infty)$. Similar remarks
apply to $\mathcal V^C_-$, of course.

Now the restriction maps $\mathcal V^C\to\mathcal V^C_{\pm}$ are continuous,
and the spaces $(\mathcal V^C_{\pm}, d_{\pm})$ are compact.

Finally, we introduce
\begin{align*}
\mathcal R^C_+(A) & = \{ \chi_{(0,\infty)}\mu : \mu\in\mathcal R^C(A) \}
\subset\mathcal V^C_+ , \\
\mathcal R^C_-(A) & = \{ \chi_{(-\infty,0)}\mu : \mu\in\mathcal R^C(A) \}
\subset\mathcal V^C_- ,
\end{align*}
and we use the same metrics $d_{\pm}$ on these spaces also.
With this setup, we now obtain
statements that are analogs of \cite[Proposition 4.1]{Remac}.
\begin{Proposition}
\label{P4.2}
Let $A\subset\R$ be a Borel set of positive measure, and fix $C>0$. Then:\\
(a) $\left( \mathcal R^C(A), d \right)$ and $\left( \mathcal R^C_{\pm}(A), d_{\pm} \right)$
are compact spaces;\\
(b) The restriction maps
\begin{align*}
\mathcal R^C(A) \to \mathcal R^C_+(A), & \quad \mu\mapsto \chi_{(0,\infty)}\mu ;\\
\mathcal R^C(A) \to \mathcal R^C_-(A), & \quad \mu\mapsto \chi_{(-\infty,0)}\mu
\end{align*}
are continuous and bijective (and thus homeomorphisms).\\
(c) The inverse map
\[
\mathcal R^C_-(A)\to \mathcal R^C(A), \quad \chi_{(-\infty,0)}\mu\mapsto\mu
\]
is (well defined, by part (b), and) uniformly continuous.
\end{Proposition}
\begin{proof}
(a) Since $\mathcal R^C(A)$ is a subspace of the compact space $\mathcal V^C$,
it suffices to show that $\mathcal R^C(A)$ is closed. This can be done exactly
as in \cite[Proof of Proposition 4.1(d)]{Remac}; we make use of Lemma \ref{L4.1}
of the present paper and Theorem 2.1, Lemma 3.2 of \cite{Remac}.

The spaces $\mathcal R^C_{\pm}(A)$ are the images of the compact space
$\mathcal R^C(A)$ under the continuous restriction maps, so these spaces are
compact, too.

(b) Continuity of the restriction maps is clear (and was already used
in the preceding paragraph). Moreover, these maps are surjective by the definition
of the spaces $\mathcal R^C_{\pm}(A)$. Injectivity follows from eq.\ \eqref{3.2}:
$\mu$ on $(0,\infty)$ determines $m_+(x,\cdot)$ for all $x>0$. Fix an $x>0$ with
$\mu(\{ x\})=0$. Since $\mu$ is reflectionless on $A$ and $|A|>0$,
we have condition \eqref{3.2} on a set of positive measure, and this lets us find
$m_-(x,\cdot)$. This $m$ function, in turn, determines $\mu$ on $(-\infty, x)$.

Finally, recall that a continuous bijection between compact metric spaces automatically
has a continuous inverse.

(c) This is an immediate consequence of parts (a) and (b).
\end{proof}
As in \cite{Remac}, the Oracle Theorem will follow by combining Proposition \ref{P4.2}
with Theorem \ref{T4.1}. In fact, in rough outline, things are rather obvious now:
Proposition \ref{P4.2} says that a reflectionless potential can be approximately
predicted if it is known on a sufficiently large interval (recall how the topologies
on the spaces $\mathcal R^C(A)$, $\mathcal R^C_{\pm}(A)$ were defined),
and Theorem \ref{T4.1} makes sure
that $S_x\mu$ is approximately reflectionless for sufficiently large $x$.

Some care must be exercised, however, if a \textit{continuous} oracle $\Delta$
is desired. The following straightforward but technical considerations prepare
for this part of the proof. We again consider the spaces $\mathcal V^C_J$ with
metrics of the type described in Section 2; in
the applications below, the interval $J$ will be bounded and open, but this is not
essential here.

In a normed space, balls $B_r(x)=\{ y: \|x-y\|<r \}$ are convex;
Lemma \ref{L4.2} below says that balls with respect to the metric \eqref{defd} enjoy
the following weaker, but analogous property.
\begin{Lemma}
\label{L4.2}
If $w_j\ge 0$, $\sum_{j=1}^m w_j = 1$ and $\mu,\nu_j\in\mathcal V^C_J$ satisfy
$d(\mu,\nu_j)<\epsilon$ with $\epsilon\le 1/4$ (say), then
\[
d\left( \mu, \sum_{j=1}^m w_j\nu_j \right) < 6\epsilon \ln \epsilon^{-1} .
\]
\end{Lemma}
\begin{proof}
Let $N=\max\{ n\in\N : 2^{n+1} \epsilon \le 1 \}$, and abbreviate $\sum w_j\nu_j=\nu$.
Since $d(\mu,\nu_j)<\epsilon$, it is clear from \eqref{defd} that if $n\le N$, then
\[
\rho_n(\mu,\nu_j) < \frac{2^n\epsilon}{1-2^n\epsilon}\le 2^{n+1}\epsilon .
\]
The definition of $\rho_n$ shows that
\[
\rho_n(\mu,\nu) \le \sum w_j \rho_n(\mu, \nu_j) ,
\]
so we obtain that
\[
\rho_n(\mu,\nu) < 2^{n+1}\epsilon \quad\quad (n\le N) .
\]
This allows us to estimate
\[
\sum_{n=1}^N 2^{-n} \frac{\rho_n(\mu,\nu)}{1+\rho_n(\mu,\nu)}
< \sum_{n=1}^N 2^{-n}\cdot2^{n+1}\epsilon = 2N\epsilon < \frac{2\epsilon\ln\epsilon^{-1}}{\ln 2} .
\]
On the other hand, we of course have that
\[
\sum_{n>N} 2^{-n} \frac{\rho_n(\mu,\nu)}{1+\rho_n(\mu,\nu)}
< \sum_{n>N} 2^{-n} = 2^{-N} < 4\epsilon \le \frac{4\epsilon\ln\epsilon^{-1}}{\ln 4},
\]
so we obtain the Lemma.
\end{proof}
\begin{proof}[Proof of Theorem \ref{OT}]
We begin by introducing some notation that will prove useful. Write
\[
J_- = (-L,0), \quad\quad J_+=(a,b) .
\]
Our goal is to (approximately) predict the restriction of $S_x\mu$ to $J_+$,
and we are given the restriction of $S_x\mu$ to $J_-$. We will use subscripts
$+$ and $-$, respectively, for such restrictions. So, for example, $\nu_+=
\chi_{J_+}\nu$, and this is now interpreted as
an element of $\mathcal V^C_{J_+}$.

Next, note that although a metric is explicitly mentioned
in Theorem \ref{OT}, by compactness, it suffices
to establish the assertion for \textit{some} metric that generates the weak $*$ topology.
We will of course want to work with the metric from \eqref{defd} and Lemma \ref{L4.2}.
More specifically, denote
this metric (on $\mathcal V^C_{J_+}$) by $d_+$. We use a similar metric $d_-$ on
$\mathcal V^C_{J_-}$; on $\mathcal V^C$, we also fix such a metric
$d$, but, in addition, we demand, as we may, that $d$ dominates $d_{\pm}$ in the
following sense: If $\mu,\nu\in\mathcal V^C$, then
\begin{equation}
\label{dpm}
d_-(\mu_-,\nu_-) \le d(\mu,\nu), \quad\quad d_+(\mu_+,\nu_+) \le d(\mu,\nu) .
\end{equation}
(The same notation, $d_{\pm}$, was used for different purposes in Proposition \ref{P4.2};
since we are not going to explicitly use those metrics here, that should not cause any confusion.)

With these preparations out of the way, the proof can now be accomplished in four steps.
Let $A\subset\R$, $|A|>0$, $\epsilon >0$, $a,b\in\R$ ($a<b$), and $C>0$ be given.

\textit{Step 1: }Use Proposition \ref{P4.2}(c) and the definition of the topologies
on $\mathcal R^C(A)$, $\mathcal R^C_-(A)$
to find $L>0$ and $\delta>0$ such that the following holds:
For $\nu,\widetilde{\nu}\in\mathcal R^C(A)$,
\begin{equation}
\label{4.6}
d_-\left( \nu_-, \widetilde{\nu}_- \right) < 5\delta \quad
\Longrightarrow\quad d_+\left( \nu_+, \widetilde{\nu}_+ \right) < \epsilon^2 .
\end{equation}
We further assume that $\delta\le\epsilon$ here. (The suspicious reader will have
noticed that it is at this point only that we can define $d_-$ and $d$.)

\textit{Step 2: }The set
\[
\mathcal R^C_{J_-}(A) := \left\{ \mu_- : \mu\in \mathcal R^C(A) \right\}
\]
is compact by Proposition \ref{P4.2} (again, this is a continuous image of a compact space).
Since $\mathcal V^C_{J_-}$ is compact, it follows that the closed $\delta$ neighborhood
\[
\overline{U}_{\delta} = \left\{ \mu_-\in \mathcal V^C_{J_-} : d_-(\mu_-,\nu_-)\le\delta
\textrm{ for some }\nu\in\mathcal R^C(A) \right\}
\]
is also compact. Therefore, there exist $\nu_1,\ldots, \nu_N\in\mathcal R^C(A)$ so that
the $2\delta$ balls about the $\nu_{j,-}$ cover $\overline{U}_{\delta}$. At these points,
we can define a preliminary version of the oracle in the obvious way as
\[
\Delta_0(\nu_{j,-}) = \nu_{j,+} .
\]
However, this will be modified in the next step.

\textit{Step 3: }We now define, for arbitrary $\sigma\in\overline{U}_{\delta}$,
\[
\Delta(\sigma) = \frac{\sum (3\delta - d_-(\sigma,\nu_{j,-}))\Delta_0(\nu_{j,-})}
{\sum (3\delta - d_-(\sigma,\nu_{j,-}))} .
\]
The sums are over those $j$ for which $d_-(\sigma,\nu_{j,-})<3\delta$.
It's easy to see that $\Delta:\overline{U}_{\delta}\to\mathcal V^C_{J_+}$
is continuous. Moreover, if $j_0\in \{ 1,\ldots, N\}$
is such that $d_-(\sigma,\nu_{j_0,-})<2\delta$, then
$d_-(\nu_{j,-},\nu_{j_0,-})<5\delta$ for all $j$ contributing to the sum. Therefore,
\eqref{4.6} shows that $d_+(\nu_{j,+},\nu_{j_0,+})<\epsilon^2$ for these $j$.
If $\epsilon>0$ was sufficiently small, then Lemma \ref{L4.2} now implies that
\begin{equation}
\label{4.7}
d_+\left( \Delta(\sigma),\nu_{j_0,+} \right) < 6\epsilon^2\ln\epsilon^{-2} < \epsilon ,
\end{equation}
say. Recall that this holds for every $j_0$ for which
$d_-(\sigma,\nu_{j_0,-})<2\delta$. Moreover, for every
$\sigma\in\overline{U}_{\delta}$, there is at least one such index $j_0$.

The oracle $\Delta$ has now been defined on $\overline{U}_{\delta}$, and this is all
we need to do the prediction. However, if a (somewhat specious)
continuous extension to all of $\mathcal V^C_{J_-}$ is desired, one can proceed
as above, by considering a suitable covering and taking convex combinations.
It is also possible, somewhat more elegantly, to just refer to the extension theorem
of Dugundji-Borsuk \cite[Ch.~II, Theorem 3.1]{BesP}.

\textit{Step 4: }In this final step, we show that $\Delta$ indeed predicts $\mu$.
Given a potential $\mu\in\mathcal V^C$ with $\Sigma_{ac}\supset A$, first of all
take $x_0$ so large that
\[
d\left( S_x\mu , \omega(\mu) \right) < \delta \quad\quad \textrm{for all }x\ge x_0 .
\]
In other words, if we fix $x\ge x_0$, we then have that
\begin{equation}
\label{4.8}
d( S_x\mu, \nu ) < \delta
\end{equation}
for some (in general: $x$ dependent) $\nu\in\omega(\mu)$. By Theorem \ref{T4.1},
$\nu\in\mathcal R^C(A)$. When we restrict to $J_{\pm}$,
then \eqref{dpm}, \eqref{4.8} imply that
\begin{equation}
\label{4.9}
d_-\left( [S_x\mu]_-, \nu_- \right) < \delta, \quad\quad
d_+ \left( [S_x\mu]_+, \nu_+ \right) < \delta .
\end{equation}
In particular, this ensures that $[S_x\mu]_-\in\overline{U}_{\delta}$, and thus
there exists a $j\in\{ 1, \ldots , N\}$ so that
\[
d_- \left( [S_x\mu]_-, \nu_{j,-} \right) < 2\delta .
\]
By \eqref{4.7},
\begin{equation}
\label{4.10}
d_+ \left( \Delta\left( [S_x\mu]_-\right) , \nu_{j,+} \right) < \epsilon .
\end{equation}
But by the triangle inequality, we also have that $d_-(\nu_-,\nu_{j,-})<3\delta$,
so \eqref{4.6} shows that
\[
d_+ \left( \nu_+, \nu_{j,+} \right) < \epsilon^2 .
\]
If this is combined with \eqref{4.9}, \eqref{4.10}, we indeed obtain that
\[
d_+ \left( \Delta\left( [S_x\mu]_-\right) , [S_x\mu]_+ \right) < \delta + \epsilon +
\epsilon^2 < 3\epsilon
\]
(say), as desired.
\end{proof}
\noindent\textit{Acknowledgment.} I thank Sergey Denisov and Barry Simon for bringing
\cite{Dencont} to my attention and Lenny Rubin for useful information on extension theorems.

\end{document}